\documentclass{amsart}

\usepackage[colorlinks,linkcolor={blue},citecolor={blue},urlcolor={red},]{hyperref}

\usepackage{amssymb, amsmath}
\usepackage{mathrsfs}
\usepackage{amscd}
\usepackage{verbatim}

\usepackage{enumerate}

\usepackage[subrefformat=parens,labelformat=parens,labelfont=rm]{subfig}
\usepackage[colorinlistoftodos,prependcaption,textsize=tiny]{todonotes}

\theoremstyle{plain}
\newtheorem{theorem}{Theorem}[section]
\theoremstyle{remark}

\theoremstyle{plain}

\newtheorem{lemma}[theorem]{Lemma}
\newtheorem{proposition}[theorem]{Proposition}

\numberwithin{equation}{section}

\def\N{{\mathbb N}}

\def\R{{\mathbb R}}
\def\C{{\mathbb C}}

\def\m2{{L^2_{\mu}(E)}}

\def\Ent{\mathsf{Ent}}

\newcommand{\g}{\gamma}

\newcommand{\ga}{\gamma_{\alpha}}

\renewcommand{\L}{\mathcal{L}}
\newcommand{\U}{\mathcal{U}}

\newcommand{\om}{\omega}
\newcommand{\supp}{\text{supp}}

\allowdisplaybreaks

\begin{document}

\author{Antonio Agresti}
\address{Department of Mathematics Guido Castelnuovo\\ Sapienza University of Rome\\ P.le A. Moro 2\\ 00185 Roma\\ Italy.} \email{agresti@mat.uniroma1.it}

\author{Paola Loreti}
\address{SBAI Department\\ Sapienza University of Rome\\ Via Antonio Scarpa, 16\\ 00161 Roma\\ Italy.} \email{paola.loreti@sbai.uniroma1.it, daniela.sforza@sbai.uniroma1.it}

\author{Daniela Sforza}
\date\today


\title[Memory in entropy decay of Ornstein-Uhlenbeck operators]{Time memory effect in entropy decay \\ of Ornstein-Uhlenbeck operators}


\begin{abstract}
We investigate the effect of memory terms on the entropy decay of the solutions to equations with Ornstein-Uhlenbeck operators. 
Our assumptions on the memory kernels include Caputo-Fabrizio operators and, more generally, the stretched exponential functions.
We establish a sharp rate decay for the entropy. 
Examples and numerical simulations are also given to illustrate the results.
%
\end{abstract}

\maketitle

\setcounter{tocdepth}{1}

\section{Introduction}
\subsection{Statement of the problem}
We consider a diffusion equation with memory for Ornstein-Uhlenbeck operator 
\begin{equation}\label{eq:Volterra_diff}
u_t(x,t)+ \int_0^t k(t-\tau) u_\tau(x,\tau)d\tau=\Delta u(x,t)- \alpha x\cdot \nabla u(x,t),
\quad
x\in \R^d,\ t>0,
\end{equation}
where $\alpha$ is a positive constant. 

The novelty of the paper consists in taking  the kernel 
$k$ in \eqref{eq:Volterra_diff} satisfying the conditions
\begin{equation}
\label{ass:kernel}
k\in W^{1,1}_{loc}(0,\infty)\cap L^{1}(0,\infty), \quad  k\ \mbox{is non-negative and  non-increasing}.
\end{equation}
The stretched exponential functions 
\begin{equation}
\label{eq:caputo_fabrizio}
k(t)=\nu e^{-t^{\beta}},\qquad \nu,\beta>0,
\end{equation}
and, in particular for $\beta=1$,
the Caputo-Fabrizio operators 
satisfy \eqref{ass:kernel}. 
The aim of this paper is to establish sharp decay estimates for the entropy of the solution $u$ to \eqref{eq:Volterra_diff}, defined as
\begin{equation}
\label{eq:def_entropy_alpha}
\Ent (u(t)):= \int_{\R^d} u\ln u\ d\ga -\Big(\int_{\R^d} u\ d\ga\Big)\ln\Big(\int_{\R^d} u\ d\ga\Big),
\end{equation} 
where $d\ga$ is a Gaussian measure on $\R^d$, that is 
$$
d\ga(x):=\Big(\frac{\alpha}{2\pi}\Big)^{\frac{d}{2}} e^{-\frac{\alpha|x|^2}{2}}dx.
$$
Moreover, in order to illustrate our achievements, examples and numerical simulations are also given
when  the integral kernel $k$ is a stretched exponential function \eqref{eq:caputo_fabrizio}
and  $k$ is a power-law kernel
\begin{equation}
\label{eq:power_law}
k(t)=\nu(1+t)^{-\beta-1},\qquad \nu,\beta>0.
\end{equation}



\subsection{Motivations}
Equations with non-local time operators of parabolic type describe several phenomena related to  heat conduction with memory and diffusion processes, see e.g. \cite{Nunziato71,EvolutionaryIntegral}. Recently, there is an increasing attention to equations of the form \eqref{eq:Volterra_diff} where $k$ is not singular. 
The Caputo-Fabrizio operators cover the case of non singular kernels in the study of equation \eqref{eq:Volterra_diff}, see \cite{CF15}.  Those operators have been used to study hysteresis phenomena in materials \cite{CF17}, diffusion processes \cite{GALLAMRRAM16}, evolution of diseases  \cite{KHB19,SS19},  Fokker-Plank  equations \cite{DSG18,FJLB18}. Further applications can be found in \cite{FDA19,T19}. 
The class of kernels that we consider in this paper, see \eqref{ass:kernel}, include  Caputo-Fabrizio operators.
Besides Caputo-Fabrizio operators, our analysis covers also the so-called stretched exponential functions \cite{MauroMauro18}, see Section \ref{s:example}. 
The 
Ornstein-Ulbenbeck operator appears in many contexts related to probability and analysis \cite{H67}.  
%
%
%
%
Entropy estimates give informations on the qualitative behaviour of the solutions to \eqref{eq:Volterra_diff}. In absence of memory ($k\equiv 0$) it is well known that the entropy decay of solutions to \eqref{eq:Volterra_diff} is related to Logarithmic Sobolev Inequality, see \cite[Chapter 5]{MarkovGeometry}. More precisely, when $k\equiv 0$ the Logarithmic Sobolev Inequality for the Gaussian measure $d\ga$ on $\R^d$ is equivalent to the following decay estimate for the entropy:
\begin{equation}
\label{eq:decay_semigroup}
\Ent(u(t))\leq e^{-2\alpha t} \Ent(u_0),\qquad t\geq 0.
\end{equation}
To our knowledge, nothing is known about entropy estimates for \eqref{eq:Volterra_diff} in the general case $k\not\equiv 0$, besides the paper \cite{KP19}, where singular kernels are considered. To conclude, we remark that the main result of this paper extends to other differential operators, see Section \ref{s:conclusion}.

\subsection{Statement of the main results}
We  consider the integro-differential equation
\begin{equation}\label{eq:Volterra_diff0}
u_t(x,t)+ \int_0^t k(t-\tau) u_\tau(x,\tau)d\tau=\Delta u(x,t)- \alpha x\cdot \nabla u(x,t),
\ \ \ 
x\in \R^d,\ t>0,
\end{equation}
with the initial condition
\begin{equation}\label{eq:cauchy-datum}
u(x,0)=u_0(x),
\qquad
x\in \R^d,
\end{equation}
under the following assumptions on the integral kernel
\begin{equation}
\label{ass:kernel1}
k\in W^{1,1}_{loc}(0,\infty)\cap L^{1}(0,\infty), \quad  k\ \mbox{is non-negative and  non-increasing}.
\end{equation}
We prove an existence result.
\begin{theorem}[Well-posedness]
\label{t:wp}
Assume that $u_0$ belongs to the domain $D(L_\alpha)$ of the Ornstein-Uhlenbeck operator. Then, there exists a unique {\bf strong solution} $u\in C^1([0,\infty);L^2(\ga))\cap C([0,\infty);D(L_\alpha))$ to \eqref{eq:Volterra_diff0}--\eqref{eq:cauchy-datum}. 

Moreover, if $u_0\in L^2(\ga)$, $u_0\geq 0$ $d\ga$-- a.e., then there exists a unique {\bf weak solution} $u\in C([0,\infty);L^2(\ga))$ to \eqref{eq:Volterra_diff0}--\eqref{eq:cauchy-datum} such that $u(\cdot,t)\geq 0$ $d\ga$-- a.e. for any $t\geq 0$. 

\end{theorem}
To show the entropy decay of solutions we have to bring in, for any $\mu>0$, the unique positive non-increasing solution  $s_{\mu}\in C^{1}([0,\infty))$ of the problem
\begin{equation}
\label{eq:s_alpha}
\dot{s}_{\mu}(t)+\int_0^t k(t-\tau)\dot{s}_{\mu}(\tau) \ d\tau+\mu s_{\mu}(t)=0,\ \ t\ge0,
\quad s_{\mu}(0)=1.
\end{equation}
\begin{theorem}[Entropy decay]
\label{t:decay}
For any $u_0\in L^2(\g_{\alpha})$, $u_0\geq 0$ $d\ga$-- a.e.,  the weak solution $u$ to \eqref{eq:Volterra_diff0}--\eqref{eq:cauchy-datum} satisfies
\begin{equation}\label{eq:dec+op}
\Ent(u(t))\leq s_{2\alpha}(t) \Ent(u_0), \qquad\forall t>0,
\end{equation}
where $s_{2\alpha}$ is the  solution of \eqref{eq:s_alpha} when $\mu=2\alpha$. 

In addition, the constant $2\alpha$ in \eqref{eq:dec+op} is optimal in the following sense: if, for some $\mu>0$, the estimate
\begin{equation*}
\Ent(u(t))\leq s_{\mu}(t) \Ent(u_0), \qquad\forall u_0\in H^1(\g_{\alpha}),\ u_0\geq 0\ d\ga-a.e.,  \ t>0,
\end{equation*}
holds, then $\mu\leq 2\alpha$. 
\end{theorem}

%
%
%

\subsection{Comparison with the case without memory}
We observe that Theorem \ref{t:decay}  gives exactly the results in \cite[Chapter 5]{MarkovGeometry} when $k\equiv 0$. Moreover, it is worth noting that
the entropy decay rate of solutions to \eqref{eq:Volterra_diff0}--\eqref{eq:cauchy-datum} is  larger than the one of the case without memory. 
Indeed, if we differentiate the function  $e^{2\alpha t} s_{2\alpha}$, thanks to \eqref{eq:s_alpha} with $\mu=2\alpha$, we obtain
\begin{equation*}
\frac{d}{dt}\big({e^{2\alpha t} s_{2\alpha}}\big)(t)=- e^{2\alpha t}\int_0^t k(t-\tau)\dot{s}_{2\alpha}(\tau)\ d\tau,\qquad
 (e^{2\alpha t} s_{2\alpha})(0)=1.
\end{equation*}
Since $k$ is non-negative and $s_{2\alpha}$ is non-increasing we have $\frac{d}{dt}\big({e^{2\alpha t} s_{2\alpha}}\big)(t)\geq 0$.
Therefore
$
e^{-2\alpha t}\leq s_{2\alpha}(t).
$
So, if we compare \eqref{eq:decay_semigroup} and \eqref{eq:dec+op}, then the claim follows.
This is consistent with the physical meaning of the memory term in \eqref{eq:Volterra_diff0}, see \cite{CF15,CF17}.

\subsection{Comparison with the literature}
Theorems \ref{t:wp} and \ref{t:decay} give a contribution to understand time memory effect in entropy decay for a large class of kernels. In literature entropy estimates for fractional  equations have been considered in \cite{KP19}. Although the problem investigated in \cite{KP19} is different from \eqref{eq:Volterra_diff}, the arguments used in the proof of Theorem \ref{t:decay} have been adapted from the results proved in \cite{KSVZ16,VZ15}.

\subsection{Plan of the paper}
The paper is divided into four sections. In Section \ref{s:example} we examine the decay rates of the entropy for \eqref{eq:Volterra_diff} for the stretched exponential and power-law kernels.  Section \ref{s:proof} is devoted to the proofs of Theorems \ref{t:wp}  and \ref{t:decay}. We also introduce some preliminary notations and results regarding
the Ornstein-Uhlenbeck operator,  integral equations and Logarithmic Sobolev Inequality. 
Lastly, in Section \ref{s:conclusion}
we suggest some possible extensions of our results.

\section{Analysis of the decay rate $s_{2\alpha}$}
\label{s:example}
In this section we examine the behaviour of the functions $s_{2\alpha}(t)$ that govern the entropy decay of the solutions to \eqref{eq:Volterra_diff0}--\eqref{eq:cauchy-datum}, see Theorem \ref{t:decay}, for some type of
kernels satisfiying \eqref{ass:kernel1}.

\subsection{Stretched exponential and power-law kernels}
\label{ex:exponential}
To study equation \eqref{eq:s_alpha} for $\mu=2\alpha$,
we  implement standard numerical methods. More precisely, fix $T>0$ and divide $[0,T]$ into $N$ steps of length $\Delta t$. Let us denote by $s_{n}$ the numerical solution of \eqref{eq:s_alpha} at time $t_{n}:=n\Delta t$, $n=0,\dots,N$. The numerical scheme is obtained by using finite differences to approximate the derivatives
\begin{equation*}
	\dot{s}(t) \simeq \frac{s_{n+1}-s_{n}}{\Delta t}
\end{equation*}
 and the composite trapezoidal formula \cite[Chapter 9]{QSS10} to approximate the integral term. Indeed,
 \begin{align*}
& 	\int_{0}^{t} k(t-\tau)\dot s(\tau) d\tau
 	 \simeq \Delta t\Big(\frac{k(t_{n})\dot s({0})+k({0})\dot s(t_{n})}{2}+
 	 \sum_{j=1}^{n-1}k(t_{n}-t_{j})\dot s(t_{j})\Big)\\
	&\simeq \Delta t\Big[\frac{1}{2}\Big(-k(t_{n})2\alpha+
	k({0})\frac{s_{n+1}-s_{n}}{\Delta t}\Big)+\sum_{j=1}^{n-1}k(t_{n}-t_{j})\frac{s_{j+1}-s_{j}}{\Delta t}\Big]
\end{align*}
where $n=0,\dots,N-1$ and we have used that $\dot{s}(0)=-2\alpha$, by \eqref{eq:s_alpha}. Inserting the above approximation in \eqref{eq:s_alpha}, we obtain the following numerical scheme 
\begin{equation}
\begin{aligned}
\label{eq:numerical_scheme}
s_{n+1} = \frac{2\Delta t}{2+k(0)\Delta t}&\Big[s_{n}\Big(\frac{1}{\Delta t}-2\alpha+\frac{k(0)}{2}\Big)
\\& +k(t_{n})2\alpha-\sum_{j=1}^{n-1}k(t_{n}-t_{j})(s_{j+1}-s_{j})\Big],
\end{aligned}
\end{equation}
where $n=0,\dots,N-1$.

We analyse the solutions of equation \eqref{eq:s_alpha} in the case of the stretched exponential functions \eqref{eq:caputo_fabrizio}, see Figure \ref{fig:stretched} below. 

\subsubsection{Stretched exponential kernels.} 
In Figure \ref{fig:stretched} we compare the behaviour of $s_{2\alpha}$ with the case $k\equiv 0$ by varying the parameters $\beta,\nu$ and $\alpha$. 
In Figure \subref*{fig:comparison} we set $\beta=1$, thus the numerical solution coincides with \eqref{eq:s_mu_exponential} and it presents a slower decay than $e^{-2t}$, which corresponds to the case $k\equiv 0$. In the remaining plots we compare the decays varying one parameter out of the above mentioned three. We observe that increasing $\beta$ and $\alpha$ we obtain a stronger decays (cf. Figures \subref*{fig:beta} and \subref*{fig:alpha}), while we have the opposite behaviour changing $\nu$  (Figure \subref*{fig:nu}).

In the special case $\beta=1$, we obtain the explicit expression for the solution. Indeed,
we study \eqref{eq:s_alpha} with $\mu=2\alpha$  and  $k(t)=\nu e^{-t}$,that is
\begin{equation*}
\dot{s}_{2\alpha}(t)+\nu\int_0^t e^{- (t-\tau)}\dot{s}_{2\alpha}(\tau)d\tau+2\alpha s_{2\alpha}(t)=0, \quad \text{ a.e. }  t>0,
\quad s_{2\alpha}(0)=1\,.
\end{equation*}
Multiplying  by $e^{ t}$, we can write
\begin{equation}\label{eq:beta1}
e^{ t}\dot{s}_{2\alpha}(t)+\nu\int_0^t e^{\tau}\dot{s}_{2\alpha}(\tau)d\tau+2\alpha e^{ t}s_{2\alpha}(t)=0 
\,.
\end{equation}
If we denote by $g(t)=e^{ t}s_{2\alpha}(t)$, then we note that $g(0)=1$, $e^{ t}\dot{s}_{2\alpha}(t)=\dot g(t)-g(t)$ and $\dot{g}(0)=1-2\alpha$. 
Therefore, the equation \eqref{eq:beta1} can be written in the form
\begin{equation*}
\dot g(t)+(2\alpha-1+\nu)g(t)- \nu\int_0^t g(\tau)d\tau-\nu=0 .
\end{equation*}
Differentiating the above equation we get
\begin{equation*}
\ddot g(t)+(2\alpha-1+\nu)\dot g(t)-\nu g(t)=0,
\end{equation*}
with initial conditions $g(0)=1$ and $\dot g(0)=1-2\alpha$.
Set
\begin{align*}
&\lambda_{\pm}=\frac{-(2\alpha-1+\nu)\pm \sqrt{(2\alpha-1+\nu)^2+4 \nu}}{2}, \\
&C_{+}=-\frac{\lambda_{-}+2\alpha-1}{\lambda_{+}-\lambda_{-}},\quad
C_{-}=\frac{\lambda_{+}+2\alpha-1}{\lambda_{+}-\lambda_{-}},
\end{align*}
we have
\begin{equation*}
g(t)=C_+ e^{\lambda_+ t} 
+C_- e^{\lambda_{-} t}.
\end{equation*}
Since $s_{2\alpha}(t)=e^{-t}g(t)$, we obtain
\begin{equation}
\label{eq:s_mu_exponential}
s_{2\alpha}(t)=C_+ e^{(\lambda_+-1)t}+C_{-} e^{(\lambda_- -1)t}, \qquad t>0.
\end{equation}
We also note that $\lambda_- - 1<-2\alpha<\lambda_+-1<0$. 

In conclusion, the expression \eqref{eq:s_mu_exponential} shows that the function $s_{2\alpha}(t)$ has an exponential behaviour, where the leading term $e^{(\lambda_+-1) t}$ depends on the kernel $k(t)=\nu e^{-t}$.

\begin{figure}[h!!]
\centering
\subfloat[][Case $\alpha=\nu=1$.]{\label{fig:comparison}
\includegraphics[width=0.47\linewidth]{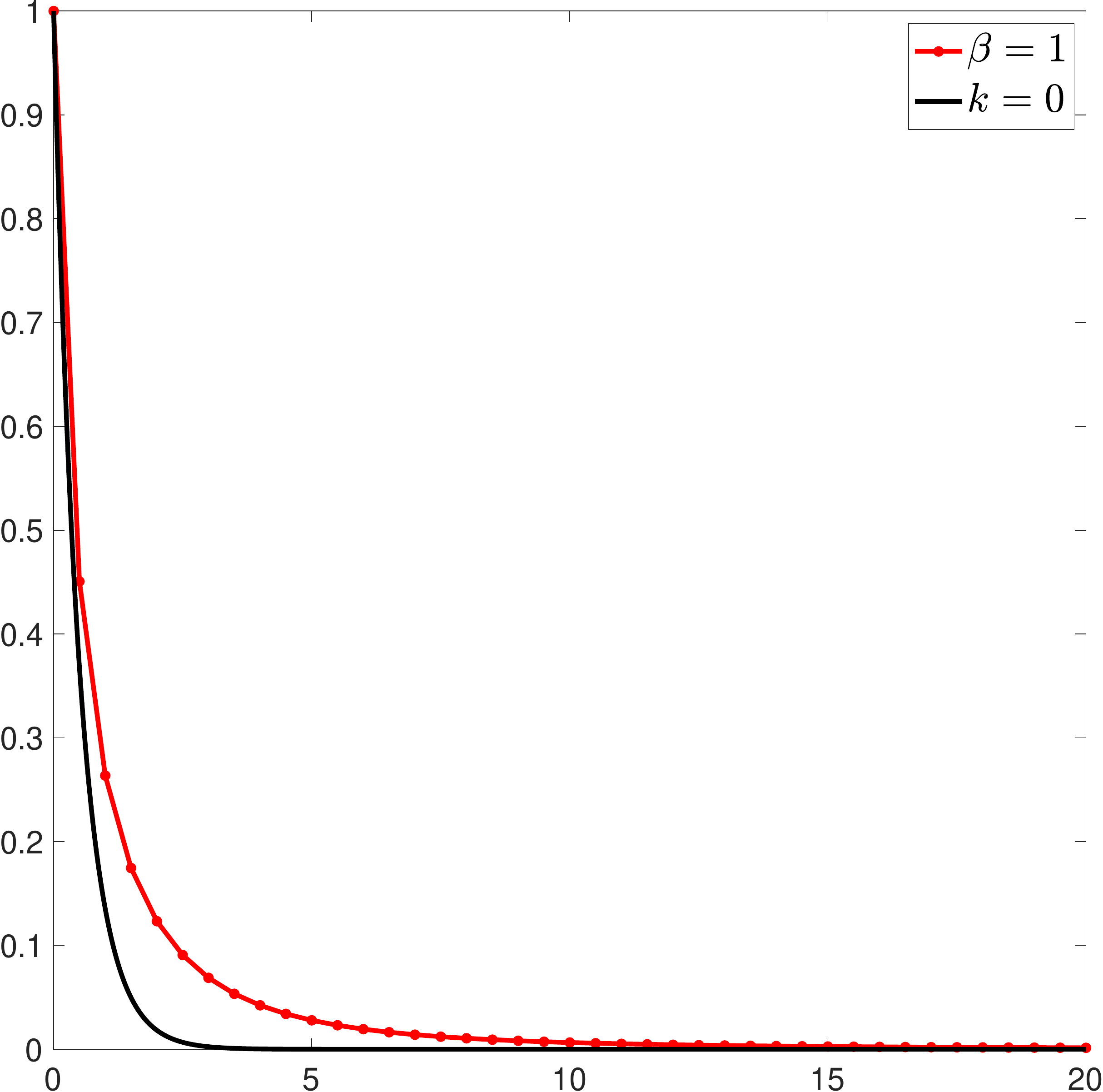}
}\,
\subfloat[][Case $\alpha=\nu=1$.]{\label{fig:beta}
\includegraphics[width=0.47\linewidth]{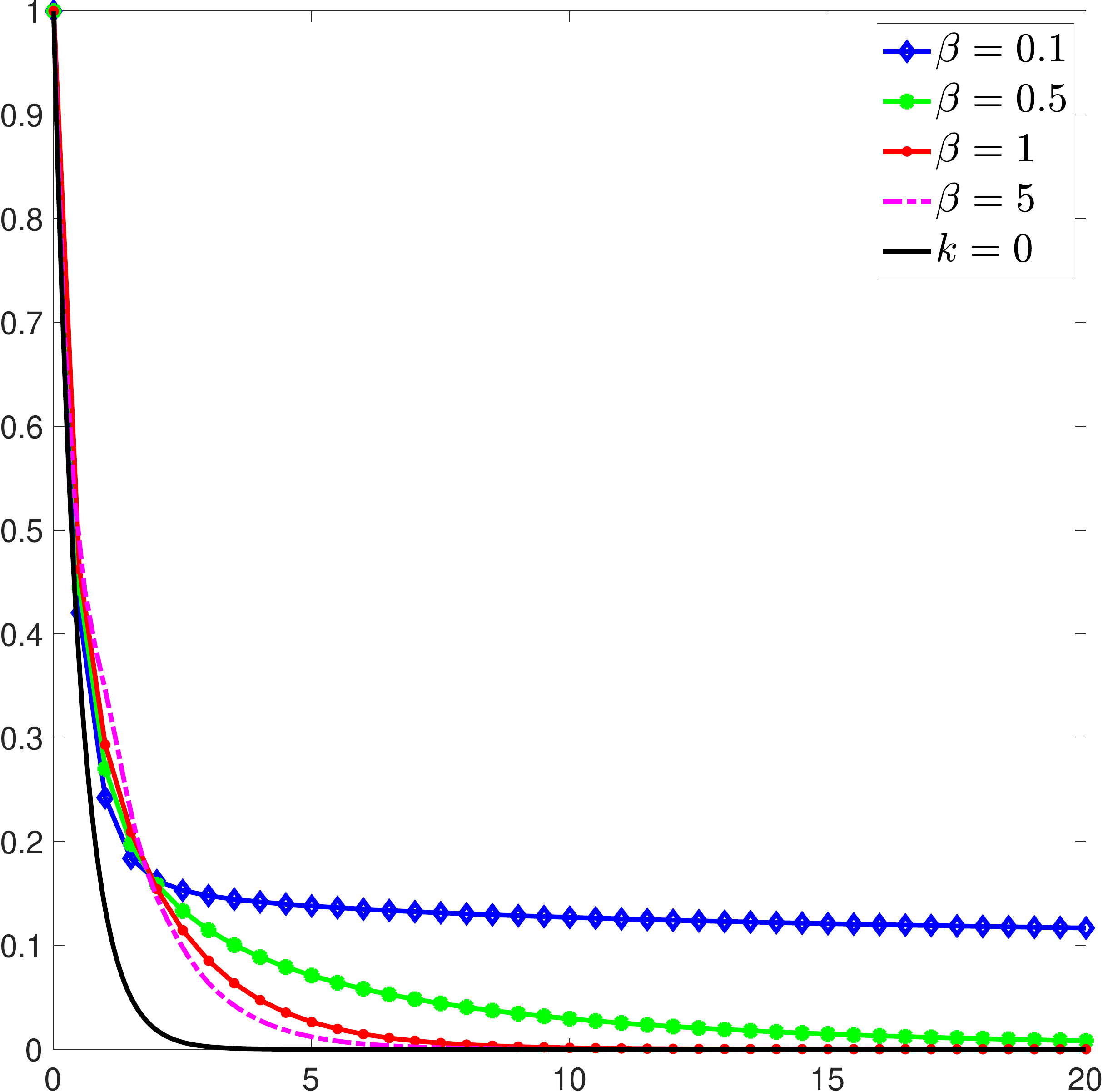}
}\\
\subfloat[][Case $\alpha=\beta=1$.]{\label{fig:nu}
\includegraphics[width=0.47\linewidth]{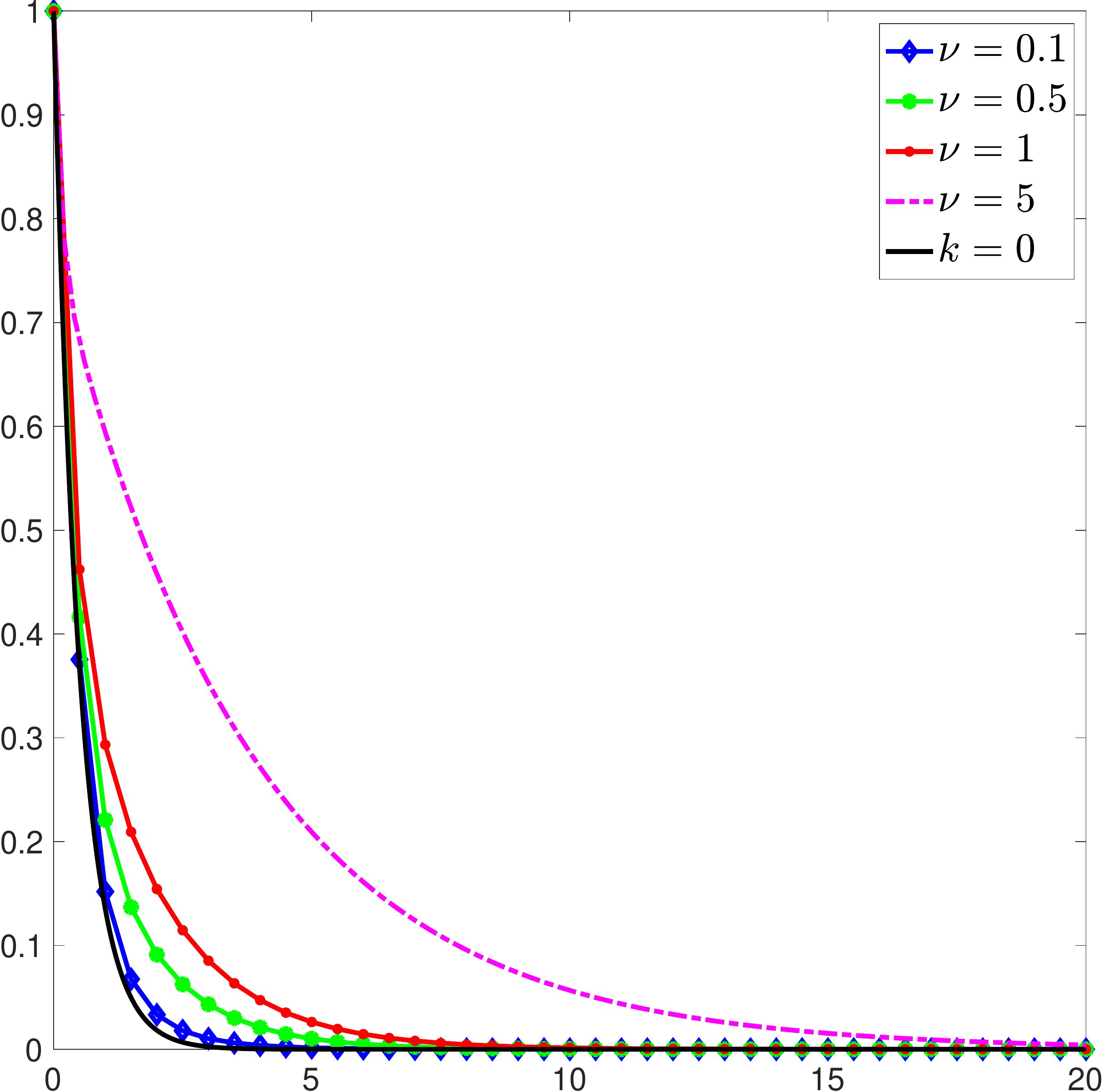}
}\,
\subfloat[][Case $\nu=\beta=1$.]{\label{fig:alpha}
\includegraphics[width=0.47\linewidth]{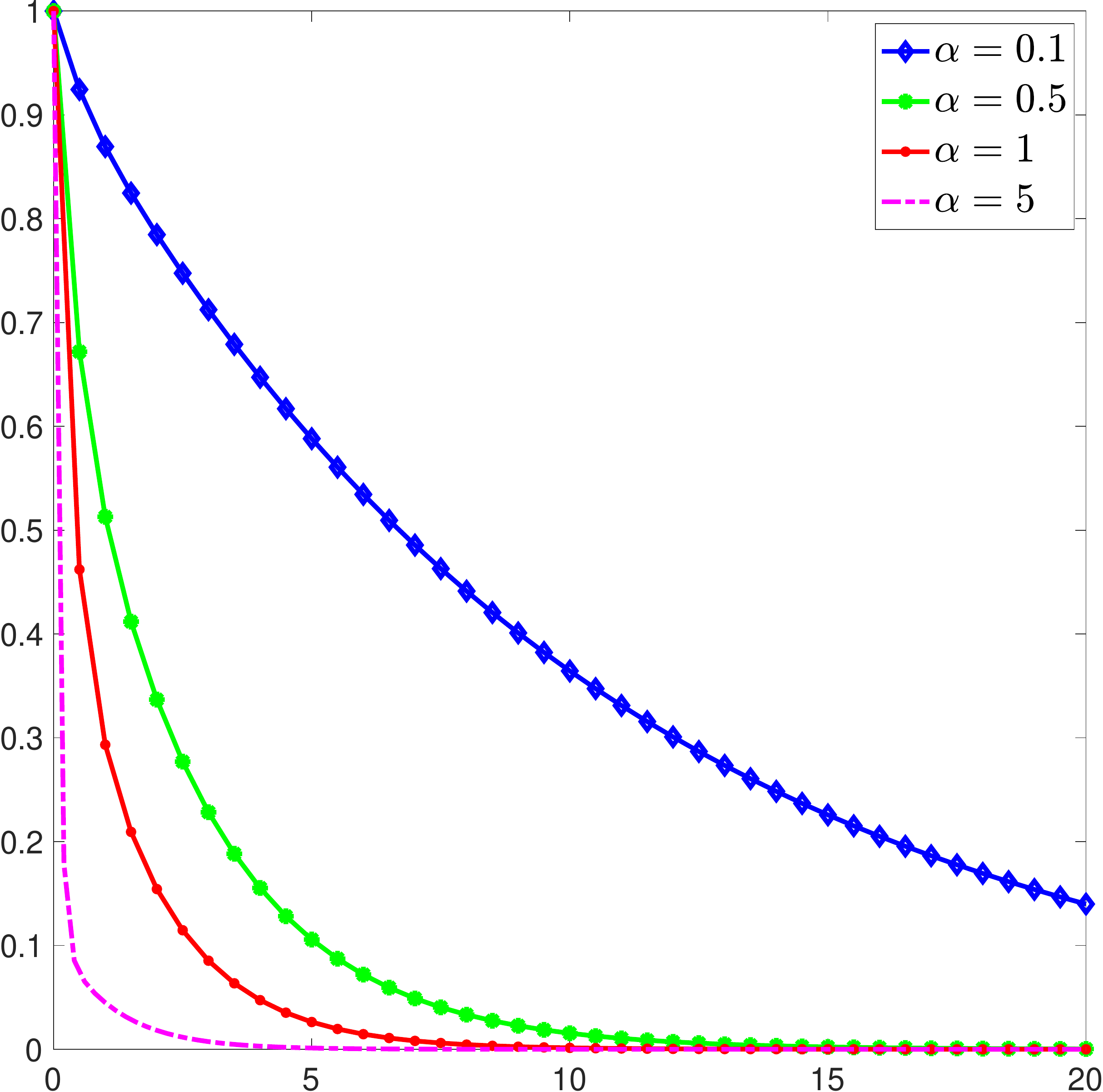}
}
\caption{Plots of $s_{2\alpha}$ for $k(t)=\nu e^{-t^{\beta}}$.}
\label{fig:stretched}
\end{figure}

\subsubsection{Power-law kernels}
We also implement the numerical scheme \eqref{eq:numerical_scheme} in the case $k=\nu(1+t)^{-\beta-1}$. As Figure \ref{fig:power_law} shows, the decay is faster with the increase of $\beta$ (Figures \subref*{fig:beta2}), while it is slower with the rise of $\nu$ (Figure \subref*{fig:nu2}).

\begin{figure}[h!]
\centering
\subfloat[][Case $\nu=\alpha=1$.]{\label{fig:beta2}
\includegraphics[width=0.47\linewidth]{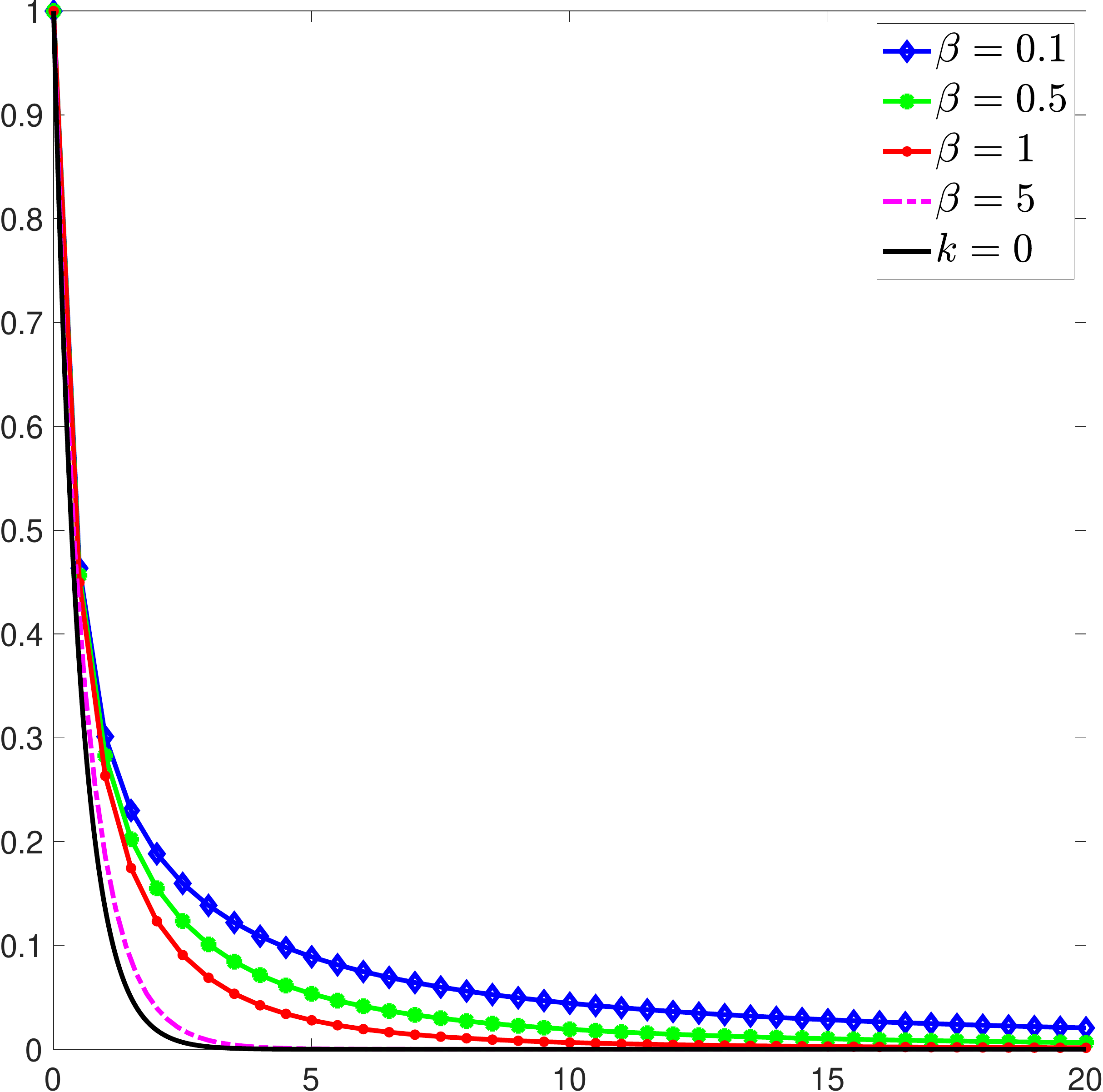}
}\,
\subfloat[][Case $\beta=\alpha=1$.]{\label{fig:nu2}
\includegraphics[width=0.47\linewidth]{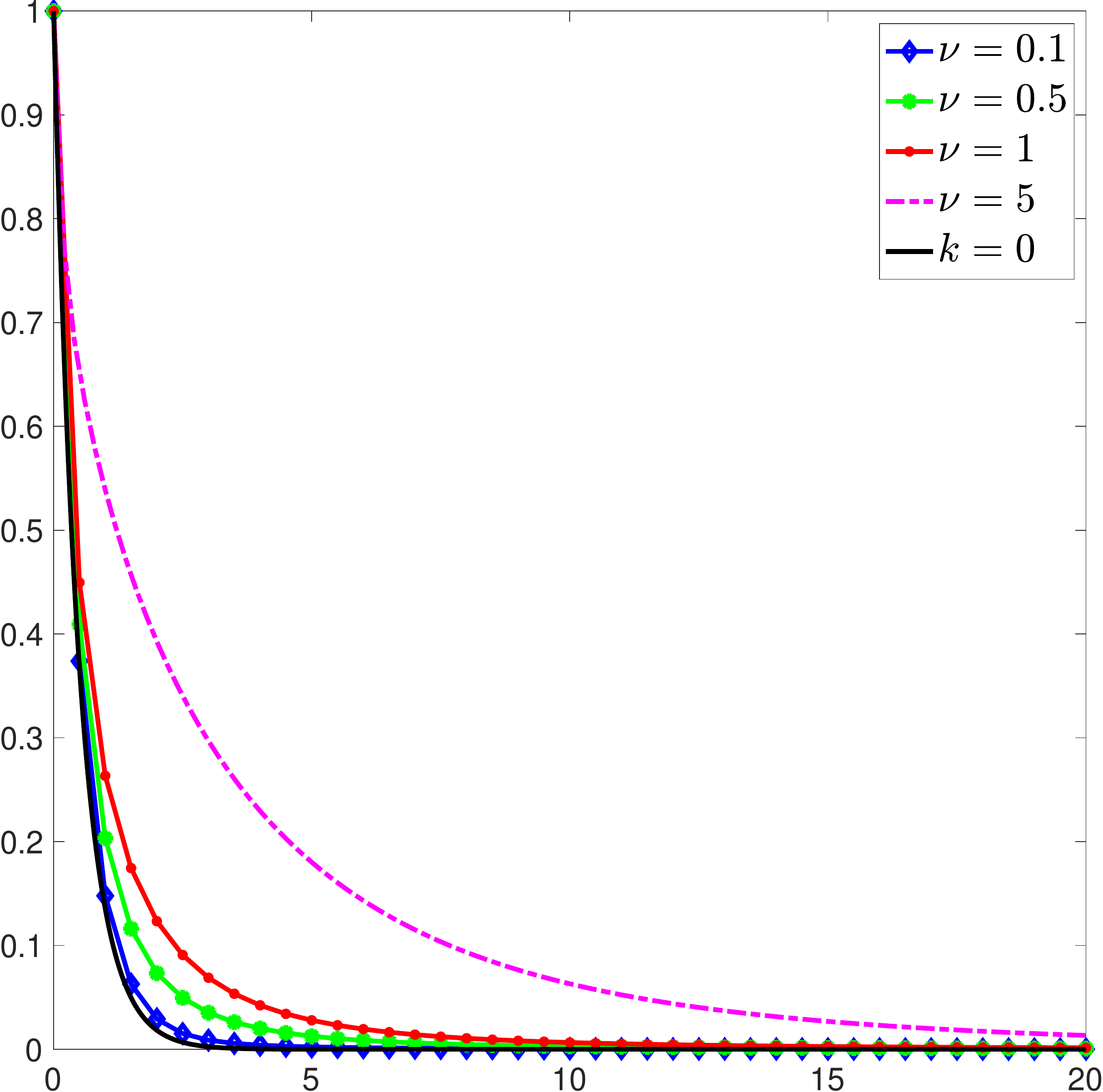}
}
\caption{Plots of $s_{2\alpha}$ for $k=\nu(1+t)^{-\beta-1}$.}
\label{fig:power_law}
\end{figure}

\section{Proof of Theorems \ref{t:wp} and \ref{t:decay}}
\label{s:proof}
To begin with, we introduce some notations and discuss some preliminary results. 
\subsection{The Ornstein-Uhlenbeck operator}
\label{ss:op_orns_uhl}
We denote by 
$$
\ga(x)=\left(\frac{\alpha}{2\pi}\right)^{\frac{d}{2}} e^{-\frac{\alpha|x|^2}{2}}
$$ 
a Gaussian distribution on $\R^d$ and by $d\ga(x)=\ga(x)dx$ the associated probability measure. For $\alpha=1$ we use the notation $\g=\ga$.
$L^2(\ga)$ is the space of measurable functions $f:\R^d\to \R$ such that $\int_{\R^d} |f|^2 d\ga<\infty$, endowed with the usual scalar product $(\cdot,\cdot)_{L^2(\ga)}$ and norm
$
\|\cdot\|_{L^2(\ga)}.
$
$H^{1}(\ga)$ denotes the space of functions $f\in L^2(\ga)$ such that $\nabla f\in L^2(\ga)$, endowed with the norm 
$$
\|f\|_{H^{1}(\ga)}:=\|f\|_{L^2(\ga)}+\|\nabla f\|_{L^2(\ga)}.
$$
There are several ways to introduce the Ornstein-Uhlenbeck operator on $L^2(\ga)$. Following \cite{Gross75},
we consider the bilinear symmetric form $\L_{\alpha}: H^{1}(\ga)\times H^{1}(\ga)\to \R$ defined by
$$
\L_{\alpha}(f,g):=\int_{\R^d} \nabla f \cdot \nabla g\ d\ga,
\qquad
f,g\in H^{1}(\ga).
$$
$\L_{\alpha}$ induces the operator $L_{\alpha}$ on $L^2(\ga)$ defined by
\begin{equation}
\label{eq:def_L_alpha}
\begin{aligned}
D(L_{\alpha})&=\{f\in H^{1}(\ga)\,:\,\Delta f-\alpha x\cdot \nabla f\in L^2(\ga)\},\\
L_{\alpha} f&=\Delta f- \alpha x\cdot \nabla f, \qquad f\in D(L_{\alpha}),
\end{aligned}
\end{equation}
that satisfies
$$
\L_{\alpha}(f,g)=-(L_{\alpha}f,g)_{L^2(\ga)},  \qquad \forall g\in H^{1}(\ga),f\in D(L_{\alpha}).
$$
$L_{\alpha}$ is the so-called Ornstein-Uhlenbeck operator. We recall that $L_{\alpha}$ is a negative self-adjoint operator, that  
generates a positive analytic semigroup $(T(t))_{t\geq 0}$ on $L^2(\ga)$, see e.g. \cite[Section 2.7.1]{MarkovGeometry}.

For completeness we state and prove an
integration by parts formula that will be useful in the sequel.

\begin{lemma}
\label{le:integration_by_parts}
Let $L_{\alpha}$ be the Ornstein-Uhlenbeck operator. The following properties hold.
\begin{itemize}
\item[(i)] 
For any $f\in D(L_{\alpha})$ there exists a sequence $\{f_k\}$ of functions belonging to $C_c^{\infty}(\R^d)$ such that $\nabla f_k\xrightarrow [{\scriptscriptstyle k\to\infty}]{}  \nabla f$ 
and $L_{\alpha}(f_k)\xrightarrow[{\scriptscriptstyle k\to\infty}]{}  L_{\alpha} f$ in $L^2(\ga)$.
\item[(ii)]  
Assume $f\in D(L_{\alpha})$, $\U\subset \R$ an open set
 and $\Phi:\U\to \R$ a $C^1$-function. For any  $g\in H^{1}(\ga)$ such that $g(x)\in \U$ $\ga$-a.e. on $\R^d$, $\Phi(g)\in L^2(\ga)$ and $\Phi'(g)\in L^{\infty}(\ga)$ we have
\begin{equation}\label{eq:integration_by_parts}
\int_{\R^d} L_{\alpha}f\ \Phi(g)d\ga= -\int_{\R^d} \Phi'(g) \nabla f	\cdot \nabla g\ d\ga.
\end{equation}
\end{itemize}
\end{lemma}

\begin{proof}
(i) The statement follows by means of the usual techniques of convolution and cut-off.

\noindent
(ii) By (i)  and the  fact that $\Phi(g)\in L^2(\ga)$ and $\Phi'(g)\in L^{\infty}(\ga)$, it is enough to prove \eqref{eq:integration_by_parts} when $f$ belongs to $ C^{\infty}_c(\R^d)$.
Indeed,
choose $R>0$ such that $\supp(f)\subset B_R:=\{x\in\R^d; |x|\leq R\}$, then
\begin{align*}
\int_{\R^d} L_{\alpha}f\ \Phi(g)d\ga
&= \int_{B_{R}} \Delta f\ \Phi(g)d\ga- \alpha\int_{B_{R}}  x\cdot \nabla f\ \Phi(g)d\ga\\
&=- \int_{B_{R}} \Phi'(g) \nabla f\cdot \nabla g\ d\ga,
\end{align*}
that is \eqref{eq:integration_by_parts}.
\end{proof}

\subsection{Evolutionary integral equations}\label{se:EIE}
\label{ss:evolutionary_integral_equations}
The purpose of this section is to recall some well-known notions and results about  integral equations.

We denote by $L_{loc}^1 (0,\infty)$ (resp. $W_{\rm loc}^{1,1}(0,\infty)$, $W_{\rm loc}^{2,1}(0,\infty)$) the space of functions
belonging to
$L^1(0,T)$ (resp. $W^{1,1}(0,T)$, $W^{2,1}(0,T)$) for any $T\in (0,\infty)$.

For any $k,f\in L_{\rm loc}^{1}(0,\infty)$ the symbol $k*f$
stands
for convolution
from $0$ to $t$, that is
$$
k * f\,(t)=\int_0^t k (t-s)f(s)\, ds,\qquad t\ge 0.
$$

As usual, the Laplace transform of a function $f\in L^1_{loc}(0,\infty)$ having sub-exponential growth (i.e. for all $\om>0$, $\int_{0}^{\infty} e^{-\om t} |f(t)|dt<\infty$)  will be denoted by 
$$
\widehat{f}(\lambda):=\int_0^\infty e^{-\lambda t}f(t)dt
\qquad
\lambda\in \C,\ \Re \lambda>0.
$$  
Classical results for integral equations
(see, e.g., \cite[Theorem 2.3.5]{GLS90})
ensure that, for any  kernel $k\in L_{loc}^1(0,\infty)$ and any $g\in
L_{loc}^{1} (0,\infty)$,  the problem
\begin{equation}\label{integral}
f(t)+k*f(t)=g(t),\qquad t\ge 0\,,
\end{equation}
admits a unique solution $f\in L_{loc}^{1}(0,\infty)$. Moreover, if $g\in
W_{loc}^{1,1} (0,\infty)$ (resp. $W_{\rm loc}^{2,1}(0,\infty)$), then we have $f\in W_{loc}^{1,1}(0,\infty)$ (resp. $W_{\rm loc}^{2,1}(0,\infty)$) too.

It is useful to recall the following result, see \cite[Lemma 1.3]{JJL}. 
 \begin{lemma}\label{le:Levin}
 If $k\in L^{1}_{loc}(0,\infty)$ is non-negative and  non-increasing and $g\in L^{1}_{loc}(0,\infty)$ is non-negative and  non-decreasing, then the solution $\varphi$ of the integral equation \eqref{integral} satisfies
 \begin{equation}
0\le \varphi(t)\le g(t)
\qquad \mbox{for a.e.}\ t\ge0\,.
\end{equation}

 \end{lemma}
 Given $b\in
L^1_{loc}(0,\infty)$, recall that $b$ is a kernel {\em  of positive type} if
\begin{equation}\label{postype}
\int_0^T b*v(t)v(t)\ dt\ge 0\,,
\qquad\mbox{for any}\ T>0\,,\, v\in L^2(0,T).
\end{equation}
If $b\in
L^\infty(0,\infty)$,  $b$ is of positive type if and only if
\begin{equation}\label{postype1}
\Re\;\widehat b(\lambda)\ge 0\,\qquad\mbox{for any}\,\lambda\in\C\,,\,\,
\Re\lambda>0
\end{equation}
(see, e.g., \cite[p.38]{EvolutionaryIntegral}).

Also, $b$ is said to be a {\em completely positive} kernel if there exists $ k\in W^{1,1}_{loc}(0,\infty)$ non-negative and non-increasing such that
\begin{equation}\label{compos}
b(t)+\int_0^t k(t-s)b(s)ds=1, \qquad  t\ge 0.
\end{equation}

\begin{lemma}\label{le:composk}
If $b$ is a completely positive kernel, then we have
\begin{itemize}
\item[(i)]
$ b\in W^{2,1}_{loc}(0,\infty)$;
\quad
$0\le b(t)\le1$
\ \ \  $\forall t\ge0$.
\item[(ii)] If $k$ is the function in \eqref{compos}, then we have
\begin{equation}\label{eq:Laplace}
\widehat{b}(\lambda)=\frac{1}{\lambda(1+\widehat{k}(\lambda))}, \qquad \Re\lambda>0.
\end{equation}
\item[(iii)] $b$ is a kernel of positive type.
\item[(iv)] For any $u_0\in\R$ and $f\in C([0,\infty))$,  $u\in C([0,\infty))$ is given by
\begin{equation}\label{eq:b*}
u(t)=u_0+b*f(t),
\quad t\ge0,
\end{equation}
if and only if $u\in C^1([0,\infty))$ and satisfies 
\begin{equation}\label{eq:b*1}
\begin{cases}
\dot u+ k*\dot u(t)=f(t), \;\; \ t\ge0, \cr
u(0)=u_0\,.
\end{cases}
\end{equation}
\end{itemize}

\end{lemma}
\begin{proof}
(i) Let $ k\in W^{1,1}_{loc}(0,\infty)$ the non-negative and non-increasing function such that \eqref{compos} holds.
We can apply Lemma \ref{le:Levin} with $g(t)\equiv1$ to obtain $0\le b(t)\le1$ for any $ t\ge0$. 

\noindent
(ii) Thanks to (i) and $0\le k(t)\le k(0)$, $t\ge0$, we have $ b, k\in L^{\infty}(0,\infty)$. Therefore, taking the Laplace transform of equation \eqref{compos} we get \begin{equation*}
\widehat{b}(\lambda)\big(1+\widehat{k}(\lambda)\big)=\frac1\lambda, \qquad \forall \Re\lambda>0,
\end{equation*}
and hence $1+\widehat{k}(\lambda)\not=0$, $\Re\lambda>0$, and \eqref{eq:Laplace} holds.

\noindent
(iii) Since $ b\in L^{\infty}(0,\infty)$ we will prove \eqref{postype1}. Indeed, from \eqref{eq:Laplace} we deduce for $\Re\lambda>0$
\begin{equation*}
\Re\widehat{b}(\lambda)
=\frac{\Re\lambda+\Re\lambda\Re\widehat{k}(\lambda)-\Im\lambda\Im\widehat{k}(\lambda)}{\big|\lambda(1+\widehat{k}(\lambda))\big|^2}
\,.
\end{equation*}
Integrating by parts, we have
\begin{multline*}
\Re\lambda\Re\widehat{k}(\lambda)=\Re\lambda\int_0^\infty e^{-\Re\lambda t}\cos(\Im\lambda t)k(t)\ dt
=-\int_0^\infty \partial_t(e^{-\Re\lambda t})\cos(\Im\lambda t)k(t)\ dt
\\
=k(0)+\Im\lambda\Im\widehat{k}(\lambda)+\int_0^\infty e^{-\Re\lambda t}\cos(\Im\lambda t)\dot k(t)\ dt
\,.
\end{multline*}
Thanks to $\dot k(t)\le0$ we note that 
\begin{equation*}
k(0)+\int_0^\infty e^{-\Re\lambda t}\cos(\Im\lambda t)\dot k(t)\ dt
=\int_0^\infty \big(e^{-\Re\lambda t}\cos(\Im\lambda t)-1\big)\dot k(t)\ dt
\ge 0,
\end{equation*}
and hence
\begin{equation*}
\Re\lambda\Re\widehat{k}(\lambda)-\Im\lambda\Im\widehat{k}(\lambda)\ge 0,
\end{equation*}
that is 
$
\Re\widehat{b}(\lambda)>0
$
for $\Re\lambda>0$.

\noindent
(iv)  If $u$ is given by
\eqref{eq:b*}, then by $k*u$, using \eqref{compos} and differentiating, we obtain
\eqref{eq:b*1}. Vice versa, if we convolve the equation in \eqref{eq:b*1} with $b$ and apply \eqref{compos}
we get
\begin{equation*}
1*\dot u(t)=b*f(t),
\end{equation*}
hence we have \eqref{eq:b*}.
\end{proof}
Let us introduce the 
functions $s_{\mu}(t)$ associated to a completely positive kernel $b$. By \cite[Proposition 4.5]{EvolutionaryIntegral}, for any $\mu>0$ there exists a unique positive and non-increasing function $s_{\mu}\in C^{1}([0,\infty))$  such that
\begin{equation}
\label{eq:s_mu_integral}
s_{\mu}(t)+\mu b*s_{\mu}(t)=1,\qquad t\ge0.
\end{equation}
Thanks to Lemma \ref{le:composk}-(iv), equation \eqref{eq:s_mu_integral} can be written as
\begin{equation}
\label{eq:s_mu}
\dot{s}_{\mu}(t)+k*\dot{s}_{\mu}(t)+\mu s_{\mu}(t)=0, \quad \ t\ge0,
\quad s_{\mu}(0)=1\,.
\end{equation}

To estimate the entropy of the solutions to \eqref{eq:Volterra_diff}, for the non-local operator $ k*\dot{u}$
we need an identity, which looks like an analogue of the chain rule, see \cite{VZ15}. 


\begin{lemma}
\label{l:IdFundamental}
Assume $k\in W_{loc}^{1,1}(0,\infty)$.
Given $U$ an open subset of $\R$, $\Phi\in C^1(U)$ and $u\in W_{loc}^{1,1}(0,\infty)$, $u(t)\in U$ on $(0,\infty)$, 
then for $t\ge0$
\begin{itemize}
\item[(i)] $\displaystyle
\begin{aligned}[t]
&\Phi'(u(t)) (k*\dot{u})(t)
\\
=&k*\Big(\frac{d}{dt}\Phi(u)\Big)(t) 
+ \big(\Phi(u(0))- \Phi(u(t))+ \Phi'(u(t))(u(t)-u(0))\big)k(t)\\
&-\int_0^t \Big(\Phi(u(t-s))-\Phi(u(t))-\Phi'(u(t))\big(u(t-s)-u(t)\big)\Big)\dot{k}(s)ds.
\end{aligned}
$
\item [(ii)] For a non-negative and  non-increasing kernel $k$, assuming also that $\Phi$ is convex on $U$,  we have 
\begin{equation}
\label{eq:inequalityConvexity}
k*\Big(\frac{d}{dt}\Phi(u)\Big)(t) \leq \Phi'(u(t)) (k*\dot{u})(t),
\quad
t\ge0.
\end{equation}
\end{itemize}

\end{lemma}


\begin{proof}
(i)
Due to the  assumptions, we have for  $t\ge0$
\begin{align*}
\frac{d}{dt}(k*u)(t)&=k*\dot{u}(t)+ k(t)u(0),\\
\frac{d}{dt}\big(k*\Phi(u)\big)(t)&=k*\Big(\frac{d}{dt} \Phi(u)\Big)(t)+k(t)\Phi(u(0)).
\end{align*}
The assertion follows by  \cite[Lemma 2.2]{VZ15} in virtue of the above identities.

\noindent (ii)
As in  \cite[Corollary 6.1]{KSVZ16}, by the convexity of $\Phi$, taking into account that $k\geq 0$ and $\dot{k}\leq 0$,  the last two terms on the right-hand side of the identity in (i) are non-negative, so \eqref{eq:inequalityConvexity} follows.
\end{proof}

We also need  a comparison result.

\begin{lemma}
\label{le:comparison}
Assume that $ k\in W^{1,1}_{loc}(0,\infty)$ is non-negative and non-increasing.  Suppose that $u,v\in W^{1,1}_{loc}(0,\infty)$ satisfy $v(0)\leq w(0)$ and  there exists a constant $C>0$ such that \begin{equation}\label{eq:comparison}
\dot{v}+ k*\dot{v} + C v\leq 0, 
\quad
\dot{w}+ k*\dot{w}+ C w\geq 0,
\ \ \text{on}\ (0,\infty).
\end{equation}
Then $v\leq w$ on $(0,\infty)$.
\end{lemma}

\begin{proof}
The idea is essentially given in \cite[Lemma 2.6]{VZ15}. 
Set $z=v-w$, we can apply \eqref{eq:inequalityConvexity} to the convex function $\Phi(y)=\frac{1}{2}y_+^2$, where $y_+:=\max\{y,0\}$, to get
\begin{equation*}
\frac{d}{dt}z_+^2+k*\left(\frac{d}{dt}z_+^2\right)(t) 
\leq 2z_+ \left(\dot{z}+k*\dot{z} \right).
\end{equation*}
By \eqref{eq:comparison} it follows 
\begin{equation*}
\dot{z}+ k*\dot{z} + C z\leq 0,
\end{equation*}
and hence
\begin{equation*}
\frac{d}{dt}z_+^2+k*\left(\frac{d}{dt}z_+^2\right)(t) + 2C\,z\,z_+\le0.
\end{equation*}
Convolving with $b$ and applying \eqref{compos1} we have
$$
z_+^2 + 2C b*(z\,z_+)\leq 0.
$$
Since by Lemma \ref{le:composk}-(i) $b$ is positive, thanks also to $z_+z=z_+^2$, it follows 
$$
z_+^2 \leq z_+^2 + 2C b*(z_+^2)\leq  0,
\quad \text{on}\ (0,\infty),
$$
whence $v\leq w$ on $(0,\infty)$. 
\end{proof}

\subsection{Entropy and Logarithmic Sobolev Inequality}
For $\alpha>0$ we denote by $d\ga$ the Gaussian measure on $\R^d$ defined as
$$
d\ga(x):=\Big(\frac{\alpha}{2\pi}\Big)^{\frac{d}{2}} e^{-\frac{\alpha|x|^2}{2}}dx,
$$
and set $d\gamma(x)=d\gamma_1(x)$.
As well known, for a non-negative measurable function $f:\R^d\to \R$ such that $\int_{\R^d} f|\ln f| d\ga<\infty$ ($0\ln 0:=0$) the entropy of $f$ is given by
\begin{equation}
\label{eq:def_entropy}
\Ent f:=\int_{\R^d} f\ln f d\ga - \Big(\int_{\R^d} f d\ga\Big)\ln \Big(\int_{\R^d} f d\ga\Big).
\end{equation}
Note that, by Jensen inequality applied to $x\ln x$, it follows that $\Ent f\geq 0$. Moreover,
$$
\Ent(c f) =c \,\Ent(f),\quad c>0.
$$

Let us recall the following Logarithmic Sobolev Inequality.
\begin{proposition}
\label{t:Log_Sob}
Let $f\in H^1(\ga)$ be. Then
\begin{equation}\label{eq:LSI}
\Ent (f^2)\leq\frac{2}{\alpha} \int_{\R^d} |\nabla f|^2 d\ga.
\end{equation}
In particular $f^2\ln( f^2)\in L^1(\ga)$.
Moreover, the constant in \eqref{eq:LSI} is optimal.
\end{proposition}

\begin{proof}
If $\alpha=1$ inequality \eqref{eq:LSI} becomes
\begin{equation}\label{eq:LSI1}
\Ent (f^2)\leq2\int_{\R^d} |\nabla f|^2 d\gamma.
\end{equation}
and the proof  can be found in \cite{Gross75}, see also \cite[Proposition 5.5.1]{MarkovGeometry}. 
In the general case $\alpha>0$, set $f_{\alpha}(x)=f(\frac x{\sqrt{\alpha}})$ we observe that 
if $f\in H^{1}(\ga)$, then $f_{\alpha}\in H^1(\g)$. Therefore, thanks also to \eqref{eq:LSI1}, we have
\begin{align*}
\Ent (f^2)
&=\int_{\R^d} f^2\ln (f^2) d\ga - \Big(\int_{\R^d} f^2d\ga \Big)\ln \Big(\int_{\R^d} f^2d\ga \Big) \\
&= \int_{\R^d} f_{\alpha}^2\ln (f_{\alpha}^2 )d\g - 
\Big(\int_{\R^d} f_{\alpha}^2 d\g \Big)\ln \Big(\int_{\R^d} f_{\alpha}^2d\g \Big)\\
&\leq 2 \int_{\R^d}|\nabla f_{\alpha}|^2 d\g 
=\frac{2}{\alpha}\int_{\R^d}\Big|\nabla f\left(\frac{x}{\sqrt{\alpha}}\right)\Big|^2 d\g
= \frac{2}{\alpha}\int_{\R^d}|\nabla f|^2 d\ga.
\end{align*}
The optimality of the constant in the general case $\alpha>0$ follows by the optimality in the case $\alpha=1$.
\end{proof}
The following result gives the formulation of the Logarithmic Sobolev Inequality in terms of the Fisher information 
$\displaystyle \int_{\R^d} \frac{|\nabla g|^2}{g} d\ga$, where $g\in H^1(\ga)$, $g\ge 0$ d$\ga$- a.e., see \cite[p. 237]{MarkovGeometry}.

\begin{lemma}
\label{l:two_log_sob}
Let $C>0$. The following assertions are equivalent.
\begin{itemize}
\item[(a)]
\label{it:Log_S} The Logarithmic Sobolev Inequality holds
$$
\Ent (f^2)\leq C  \int_{\R^d} |\nabla f|^2 d\ga,
\qquad \text{for any}\ f\in H^{1}(\ga).
$$
\item[(b)]
\label{it:Log_I} The Entropy-Fisher Information Inequality holds 
\begin{equation*}
\Ent( g)\leq \frac{C}{4}\int_{\R^d} \frac{|\nabla g|^2}{g} d\ga,
\qquad \text{for any}\ g\in H^1(\ga), \  g\geq 0\ d\ga- a.e..
\end{equation*}
\end{itemize}
\end{lemma}

%


\subsection{Proof of Theorem \ref{t:wp}.}
\label{s:entropy_decay0}
Here we establish the well-posedness  of the integro-differential problem
\begin{equation}
\label{eq:Volterra_diff1}
\begin{cases}
\dot u(t)+ k*\dot u(t)=L_{\alpha}u(t), \;\; t>0 \cr
u(0)=u_0\,.
\end{cases}
\end{equation}
where the kernel 
$k$ satisfies the conditions
\begin{equation}
\label{ass:kernel2}
k\in W^{1,1}_{loc}(0,\infty)\cap L^{1}(0,\infty), \quad  k\ \mbox{is non-negative and  non-increasing},
\end{equation}
and $L_{\alpha}$ is the Ornstein-Uhlenbeck operator defined by \eqref{eq:def_L_alpha}.
\begin{proof}[Proof of Theorem \ref{t:wp}]
Due to the assumption \eqref{ass:kernel2} on the kernel $k$,  the unique solution $b\in W^{2,1}_{loc}(0,\infty)$  of the integral equation 
\begin{equation}\label{compos1}
b(t)+\int_0^t k(t-s)b(s)ds=1, \qquad  t> 0,
\end{equation}
is a completely positive kernel, see Section \ref{se:EIE}.
By Lemma \ref{le:composk}-(iv) for any $u_0\in D(L_{\alpha})$ we have that $u\in C^1([0,\infty);L^2(\ga))\cap C([0,\infty);D(L_{\alpha}))$ is a solution of \eqref{eq:Volterra_diff1} if and only if $u\in C([0,\infty);D(L_{\alpha}))$ is the solution of the integral equation
\begin{equation}
\label{eq:Volterra_integral}
u(t)=u_0+\int_0^t b(t-s)L_{\alpha}u(s)ds, \qquad t\ge0.
\end{equation} 
Therefore, to solve \eqref{eq:Volterra_diff1}  it is sufficient to prove the well-posedness for \eqref{eq:Volterra_integral}. 
To this end, we show that there exists the resolvent for \eqref{eq:Volterra_integral}, that is
a family $\{S(t)\}_{t\geq 0}$  of linear bounded operators in $L^2(\ga)$  such that
\begin{enumerate}[{\rm(1)}]
\item $S(0)=I$ and  for  $u_0\in L^{2}(\ga)$ the map $t\mapsto S(t)u_0$ is  continuous;
\item for  $u_0\in D(L_{\alpha})$ and $t\geq 0$, one has $S(t) u_0\in D(L_{\alpha})$, $L_{\alpha}S(t) u_0= S(t) L_{\alpha}u_0$ and
\begin{equation}
\label{eq:Volterra_integral1}
S(t)u_0=u_0+\int_0^t b(t-s)L_{\alpha}S(s)u_0ds, \qquad t\ge0.
\end{equation} 
\end{enumerate}

First, we note that by Lemma \ref{le:composk}-(iii) $b$ is a kernel of positive type. Since $L_{\alpha}$ generates an analytic semigroup (see Subsection \ref{ss:op_orns_uhl}), we can apply \cite[Corollary 3.1]{EvolutionaryIntegral} to have that equation \eqref{eq:Volterra_integral} is parabolic.
Moreover, in order to apply \cite[Theorem 3.1]{EvolutionaryIntegral}, we have to show that $b$ is 1-regular, i.e. there exists $C>0$ such that $|\lambda\widehat{b}'(\lambda)|\leq C|\widehat{b}(\lambda)|$ for all $\Re \lambda>0$.
Indeed, thanks to  \eqref{eq:Laplace} we have
\begin{equation*}
 \frac{ \lambda\widehat{b}'(\lambda)}{\widehat{b}(\lambda)} 
= 
-\frac{1+(\lambda \widehat{k}(\lambda))'}{1+\widehat{k}(\lambda)}.
\end{equation*}
Now, also by an integration by parts we get 
\begin{equation*}
(\lambda \widehat{k}(\lambda))'=\widehat{k}(\lambda)-\lambda\int_0^\infty e^{-\lambda t}tk(t) dt
=
-\widehat{t\dot{k}}(\lambda)
\,,
\end{equation*}
and hence
%
$$
 \frac{ \lambda\widehat{b}'(\lambda)}{\widehat{b}(\lambda)} 
=\frac{\widehat{t\dot{k}}(\lambda)-1}{1+\widehat{k}(\lambda)}.
$$
To prove the boundedness of the right hand-side, thanks to $k\in L^1(0,\infty)$, by Riemann-Lebesgue lemma we have $\widehat{k}(\lambda)\to 0$ as $|\lambda|\to \infty$. This implies that 
$1+\widehat{k}(\lambda)$ is bounded from below on $\{\Re \lambda>0\}$.
In addition, integrating by parts we get
\begin{multline*}
|\widehat{t\dot{k}}(\lambda)|\le-\int_0^\infty e^{-\Re \lambda t}t\dot{k}(t)\ dt
\\
=-\Re \lambda\int_0^\infty e^{-\Re \lambda t}t k(t)\ dt+\int_0^\infty e^{-\Re \lambda t} k(t)\ dt\le\int_0^\infty k(t)\ dt
\quad \forall \;\Re \lambda>0.
\end{multline*}
Therefore we have that $b$ is 1-regular. By Theorem \cite[Theorem 3.1]{EvolutionaryIntegral} we deduce the existence of the resolvent for the integral equation \eqref{eq:Volterra_integral}, that is
a family $\{S(t)\}_{t\geq 0}$  of linear bounded operators in $L^2(\ga)$  satisfying the conditions $(1)-(2)$.
 In particular, for any $u_0\in D(L_{\alpha})$ the function $S(t)u_0$ is the solution  of \eqref{eq:Volterra_integral}, and hence $S(t)u_0$ is the strong solution of \eqref{eq:Volterra_diff1}. 

Moreover, if $u_0\in L^2(\ga)$ $S(t)u_0$ is the weak solution of \eqref{eq:Volterra_diff1}, since
\begin{equation*}
S(t)u_0=\lim_{k\to\infty}S(t)u_{0k}\quad \text{in}\ L^2(\ga),
\end{equation*}
for any sequence $\{u_{0k}\}$ in $D(L_{\alpha})$ such that $u_{0k}\xrightarrow[k]{} u_{0}$ in $L^2(\ga)$.



In addition, if we assume $u_0\geq 0$  $d\ga$-- a.e., since $b$ is a completely positive kernel and $L_{\alpha}$ generates a positive semigroup on $L^{2}(\ga)$, then by \cite[Theorem 5]{Pruss87} we have
$S(t)u_0\geq 0$ $d\ga$-- a.e., for any $t\geq 0$ . 
\end{proof}

\subsection{Proof of Theorem \ref{t:decay}}
In this subsection we show a sharp rate decay for the entropy of the solutions to problem \eqref{eq:Volterra_diff1} with the integral kernel $k$ satisfying \eqref{ass:kernel2}.

To prove the statement we need the following two lemmas.

\begin{lemma}\label{le:positivity}
For any $u_0\in\  L^2(\gamma_\alpha)$, $u_0\geq \varepsilon>0$ $d\ga$-- a.e., the weak solution $u$ to problem \eqref{eq:Volterra_diff1} satisfies 
$u(t)\geq \varepsilon$ $d\ga$-- a.e. for any $t\ge0$.
\end{lemma}
\begin{proof}
The assertion follows from Theorem \ref{t:wp}, taking into account that the constant $\varepsilon$ is the unique solution to problem \eqref{eq:Volterra_diff1} when the initial condition is $\varepsilon$.
\end{proof}

\begin{lemma}[Invariance]
\label{l:conservation}
Let $u_0\in L^2(\ga)$. Then, the weak solution $u$ to problem \eqref{eq:Volterra_diff1} satisfies 
\begin{equation}\label{eq:invariance}
\int_{\R^d} u(t)\,d\ga=\int_{\R^d} u_0d\ga,
\quad\text{for any}\ t\ge0.
\end{equation}
\end{lemma}

\begin{proof}
First, we consider $u_0\in D(L_\alpha)$. By Theorem \ref{t:wp} $u$ is the strong solution to problem \eqref{eq:Volterra_diff1}.
Integrating the equation in \eqref{eq:Volterra_diff1} over $\R^d$, one has
\begin{equation*}
\frac{d}{dt}\int_{\R^d} u(t)d\ga+k*\left(\frac{d}{dt}\int_{\R^d} u(t) d\ga\right)= \int_{\R^d} L_{\alpha} u(t) d\ga.
\end{equation*}
Applying Lemma \ref{le:integration_by_parts}-(ii) with $\Phi\equiv1$ we get
\begin{equation*}
\int_{\R^d} L_{\alpha} u(t)d\ga=0,
\end{equation*}
and hence
\begin{equation*}
\frac{d}{dt}\int_{\R^d} u(t)d\ga+k*\left(\frac{d}{dt}\int_{\R^d} u(t) d\ga\right)= 0.
\end{equation*}
Thanks to the uniqueness of the solutions of integral equations \eqref{integral}, we have
\begin{equation*}
\frac{d}{dt}\int_{\R^d} u(t)d\ga\equiv 0,
\end{equation*}
that is \eqref{eq:invariance}.

The general assertion for $u_0\in L^2(\ga)$ follows by means of approximation arguments.
\end{proof}

\begin{proof}[Proof of Theorem \ref{t:decay}]
First, we prove the statement assuming the initial datum $u_0$ more regular, that is
\begin{equation}
\label{eq:u_0_smoother}
u_0\in D(L_{\alpha}), \qquad u_0\geq \varepsilon \ \ d\ga-\text{a.e.}.
\end{equation}
By Theorem  \ref{t:wp} problem \eqref{eq:Volterra_diff1} admits a unique strong solution $u$. Moreover, thanks to  Lemma \ref{le:positivity} one has $u(t)\geq\varepsilon$ $d\ga$-- a.e. for any $t\geq 0$. Therefore, we can apply inequality \eqref{eq:inequalityConvexity} with $\Phi(\tau)=\tau\log (\tau)$, $\tau>0$, to get 
\begin{equation*}
\frac{d}{dt}\Phi(u(t))+k*\Big(\frac{d}{dt}\Phi(u)\Big)(t)\leq \Phi'(u(t)) \big(\dot{u}+k*\dot{u}\big)(t).
\end{equation*}
Integrating the above inequality, thanks also to the equation in \eqref{eq:Volterra_diff1}, we obtain
\begin{equation*}
\begin{aligned}
&\int_{\R^d}  \frac{d}{dt}\Phi(u(t))+k*\Big(\frac{d}{dt}\Phi(u)\Big)(t)\ d\g_{\alpha}\\
&\leq \int_{\R^d} \Phi'(u(t)) \big(\dot{u}+k*\dot{u}\big)(t)\ d\g_{\alpha}
=\int_{\R^d}\Phi'(u(t))L_{\alpha} u(t)\ d\ga.
\end{aligned}
\end{equation*}
Since $\Phi'(u(t))=\ln u(t)+1\in L^2(\ga)$ and $\Phi''(u(t))=\frac{1}{u(t)}\in L^{\infty}(\ga)$, one can apply Lemma \ref{le:integration_by_parts}-(ii) to have
\begin{equation}\label{eq:intpart2}
\int_{\R^d}\Phi'(u(t))L_{\alpha} u(t)\ d\ga
=-\int_{\R^d}\frac{|\nabla u(t)|^2}{u(t)} d\ga,
\end{equation}
and hence
\begin{equation}
\label{eq:identity_1}
\int_{\R^d}  \frac{d}{dt}\Phi(u(t))+k*\Big(\frac{d}{dt}\Phi(u)\Big)(t)\ d\g_{\alpha}
\leq -\int_{\R^d}\frac{|\nabla u(t)|^2}{u(t)} d\ga.
\end{equation}
By \eqref{eq:LSI} applied to the function $\sqrt{u(t)}$ we have
\begin{equation*}
-\int_{\R^d}\frac{|\nabla u(t)|^2}{u(t)} d\ga\le - 2\alpha \Ent (u(t)).
\end{equation*}
Combining the above inequality with \eqref{eq:identity_1} one has
$$
\int_{\R^d}  \frac{d}{dt}\Phi(u(t))+k*\Big(\frac{d}{dt}\Phi(u)\Big)(t)\ d\g_{\alpha} \leq - 2\alpha \Ent (u(t)).
$$
Since 
\begin{equation*}
\Ent(u(t))=\int_{\R^d} \Phi(u(t))d\ga-\Phi\Big(\int_{\R^d} u(t)d\ga\Big)
\end{equation*}
and by Lemma \ref{l:conservation} the function $\int_{\R^d} u(t)d\ga$ is constant, we have
\begin{equation}\label{eq:Dent}
\frac{d}{dt}\Ent(u(t))=\int_{\R^d}  \frac{d}{dt}\Phi(u(t)).
\end{equation}
Therefore 
\begin{equation*}
\frac{d}{dt}\Ent(u(t))+k*\Big(\frac{d}{dt}\Ent(u)\Big)(t)+ 2\alpha\Ent (u(t))\leq  0.
\end{equation*}
Finally, taking into account \eqref{eq:s_mu} for $\mu=2\alpha$, that is 
\begin{equation*}
\dot{s}_{2\alpha}(t)+k*\dot{s}_{2\alpha}(t)+2\alpha s_{2\alpha}(t)=0, \quad
\quad s_{2\alpha}(0)=1\,,
\end{equation*}
we can apply Lemma \ref{le:comparison} to obtain inequality \eqref{eq:dec+op} for any $u_0$ satisfying \eqref{eq:u_0_smoother}.

In the general case we consider $u_0\in L^2(\ga)$, $u_0\geq 0$ $d\ga$-- a.e., and $u$  the weak solution to problem \eqref{eq:Volterra_diff1}. By means of the usual techniques of convolution and cut-off
we can construct a sequence $\{u_{0k}\}$ of functions belonging to $C^\infty_c(\R^d)$ such that 
\begin{equation*}
u_{0k}\geq 0\ d\ga-a.e.\quad \ \text{and} \quad u_{0k}\xrightarrow[k\to\infty]{} u_0 \quad \text{ in }L^2(\ga).
\end{equation*}
Since $u_{0k}+\frac1k$ satisfy  \eqref{eq:u_0_smoother}, denoted by $u_k$ the strong solution to problem \eqref{eq:Volterra_diff1} with initial datum $u_{0k}+\frac1k$, we have
\begin{equation}\label{eq:entk}
\Ent(u_k(t)) 
\leq s_{2\alpha}(t)\Ent\Big(u_{0k}+\frac1k\Big),
\qquad k\in\N.
\end{equation}
Thanks to $u_k(t)\xrightarrow[k]{} u(t)$ in $L^2(\ga)$, up to extract a subsequence, we can assume that $u_k(t)\xrightarrow[k]{} u(t)$ d$\ga$-- a.e. and  $|u_k(t)|\le w(t)$, with $w(t)\in L^2(\ga)$.
Since for some $C>0$ one has $\tau|\ln \tau|\leq C(1+\tau^2)$, $\tau>0$, we can apply Lebesgue dominated convergence theorem to get
\begin{equation*}
\lim_{k\to \infty}\Ent(u_k(t)) =\Ent(u(t)).
\end{equation*}
Similarly, applying again Lebesgue dominated convergence theorem, we also have
\begin{equation*}
\lim_{k\to \infty}\Ent\Big(u_{0k}+\frac1k\Big) =\Ent(u_{0}),
\end{equation*}
and hence, letting $k\to \infty$ in \eqref{eq:entk}, we obtain that inequality \eqref{eq:dec+op} holds.

To prove the optimality of the constant, we assume that, for $u_0$ satisfying \eqref{eq:u_0_smoother} and some $\mu>0$, we have
\begin{equation}
\label{eq:entropy_decay_beta}
\Ent(u(t))\leq s_{\mu}(t)\Ent (u_0),\qquad t\ge0.
\end{equation}
Computing \eqref{eq:Dent} at $t=0$, thanks also to \eqref{eq:Volterra_diff1} and \eqref{eq:intpart2} for $t=0$, one obtains,
\begin{equation}
\label{eq:t_2_step_1}
\frac{d}{dt}\Ent(u(t))\Big|_{t=0}
=\int_{\R^d} \Phi'(u_0)L_{\alpha} u_0d\ga=-\int_{\R^d} \frac{|\nabla u_0|^2}{u_0}d\ga.
\end{equation}
To estimate the left-hand side of \eqref{eq:t_2_step_1}, we note that by \eqref{eq:entropy_decay_beta} it follows
$$
\Ent (u(t))-\Ent(u_0)\leq \big(s_{\mu}(t)-1\big)\Ent(u_0),
$$
and hence, dividing for $t>0$ and sending $t\downarrow 0$, we obtain
\begin{equation}
\label{eq:t_2_decay_entropy}
\frac{d}{dt}\Ent(u(t))\Big|_{t=0}\leq \dot{s}_{\mu}(0) \Ent(u_0).
\end{equation}
Combining \eqref{eq:t_2_step_1} with \eqref{eq:t_2_decay_entropy} and taking into account that $\dot{s}_{\mu}(0)=-\mu $, see \eqref{eq:s_mu}, we get 
\begin{equation}\label{eq:step_ent_C}
\Ent(u_0)
\leq \frac1\mu\int_{\R^d} \frac{|\nabla u_0|^2}{u_0}d\ga,
\end{equation}
that is, the Entropy-Fisher Information Inequality holds for $u_0$ satisfying \eqref{eq:u_0_smoother}.
To apply Lemma \ref{l:two_log_sob} we have to prove \eqref{eq:step_ent_C} for any $u_0\in H^1(\ga)$.
To this end, first we fix $u_0\in D(L_{\alpha})$, $u_0\geq 0$ d$\ga$-- a.e.. Since \eqref{eq:step_ent_C} holds for $u_0+\frac{1}{k}$,  $k\in\N$, we have
$$
\Ent\Big(u_0+\frac{1}{k}\Big)\leq 
\frac{1}{\mu} \int_{\R^d} \frac{|\nabla u_0|^2}{u_0+\frac{1}{k}}\ d\ga\leq
\frac{1}{\mu} \int_{\R^d} \frac{|\nabla u_0|^2}{u_0}\ d\ga.
$$
By Lebesgue dominated convergence theorem, letting $k\to \infty$ in the above inequality we obtain \eqref{eq:step_ent_C}. Using again usual approximation arguments we deduce that \eqref{eq:step_ent_C} also holds for any $u_0\in H^1(\ga)$, $u_0\geq 0$ d$\ga$-- a.e.. Finally, we are able to apply
 Lemma \ref{l:two_log_sob}: the Logarithmic Sobolev Inequality holds with constant $\frac4\mu$. Therefore, since the constant $\frac2\alpha$ in \eqref{eq:LSI} is optimal, then we get $\frac2\alpha\le\frac4\mu$, that is $\mu\le2\alpha$.

\end{proof}

\section{Conclusions and extensions}
\label{s:conclusion}
In this article we study the effect of a time memory on the entropy decay of solutions to \eqref{eq:Volterra_diff}.
Our main results concern the well-posedness and optimal entropy decay, see Theorems \ref{t:wp} and \ref{t:decay}. Our assumption  \eqref{ass:kernel} on $k$ allows us to consider the stretched exponential functions \eqref{eq:caputo_fabrizio}, Caputo-Fabrizio operators and power-law kernels \eqref{eq:power_law}. Theorem \ref{t:decay} shows that the entropy decay of solutions to \eqref{eq:Volterra_diff} is governed by the function $s_{2\alpha}$, which depends on the kernel $k$, because $s_{2\alpha}$ is  the solution of the problem
\begin{equation}\label{se4:rates}
\dot{s}_{2\alpha}(t)+k*\dot{s}_{2\alpha}(t)+2\alpha s_{2\alpha}(t)=0, \quad
\quad s_{2\alpha}(0)=1\,.
\end{equation}
In Section \ref{s:example}, we  explicitly compute the solution $s_{2\alpha}$ of \eqref{se4:rates} when $k(t)=\nu e^{-t}$, that is the case of Caputo-Fabrizio operators. For general stretched exponential and power-law kernels we implement numerical schemes to examine the behaviuor of $s_{2\alpha}$. 
As Figures \ref{fig:stretched} and \ref{fig:power_law} show,  the effect of the memory in \eqref{eq:Volterra_diff} weakens the decay of the entropy with respect to the case without memory $k\equiv 0$, in accordance with the physical behaviour of some materials, see \cite{CF15}. 

The methods used in Section \ref{s:proof} seem flexible enough to study \eqref{eq:Volterra_diff} in the case the Ornstein-Uhlenbeck operator is replaced by the operator
$
\Delta -\nabla W \cdot \nabla $ where $W$ is a potential.
The latter type of operators and the relative Logarithmic  Sobolev Inequality have been considered in \cite{L99} under suitable assumptions on the potential $W$. In this paper we consider the case $W(x)=\frac{\alpha}{2}|x|^2$. 

Another possible extension is the study of the decay of a $\Phi$-entropy defined as
\begin{equation}\label{eq:entphi}
\Ent_{\Phi}f:=\int_{\R^d} \Phi(f)d\ga  - \Phi\Big(\int_{\R^d} f d\ga\Big),
\end{equation}
where $\Phi:\U\to \R$ and $f $ takes its values in $\U$, for details we refer to \cite[Section 7.6]{MarkovGeometry}. 
In the case $\Phi(\tau)=\tau\ln \tau$ and $\U=(0,\infty)$ the definition \eqref{eq:entphi} coincide with \eqref{eq:def_entropy_alpha}.


\def\polhk#1{\setbox0=\hbox{#1}{\ooalign{\hidewidth
  \lower1.5ex\hbox{`}\hidewidth\crcr\unhbox0}}} \def\cprime{$'$}

\end{document}